\documentclass[twoside, 11pt]{article}
\usepackage{mathrsfs}
\usepackage{amssymb, amsmath, mathrsfs, amsthm}
\usepackage{graphicx}
\usepackage{color}
\usepackage[top=2cm, bottom=2cm, left=2.3cm, right=2.3cm]{geometry}
\usepackage{float, caption, subcaption}
    \usepackage{diagbox}
\input{epsf}

\newcommand{\bd}{\begin{description}}
\newcommand{\ed}{\end{description}}
\newcommand{\bi}{\begin{itemize}}
\newcommand{\ei}{\end{itemize}}
\newcommand{\be}{\begin{enumerate}}
\newcommand{\ee}{\end{enumerate}}
\newcommand{\beq}{\begin{equation}}
\newcommand{\eeq}{\end{equation}}
\newcommand{\beqs}{\begin{eqnarray*}}
\newcommand{\eeqs}{\end{eqnarray*}}

\definecolor{DarkGreen}{rgb}{0.2, 0.6, 0.3}

\newtheorem{theorem}{Theorem}[section]

\newtheorem{lemma}{Lemma}[section]
\newtheorem{definition}{Definition}
\newtheorem{corollary}[theorem]{Corollary}
\newtheorem{case}{Case}

\newtheorem{claim}{Claim}

\newtheorem{proposition}{Proposition}[section]

\newtheorem{example}{Example}

\newtheorem{observation}{Observation}[section]
\usepackage{multicol}
\usepackage[noend]{algpseudocode}
\usepackage{algorithmicx,algorithm}
\usepackage{algorithm}
\usepackage{algorithmicx}
\usepackage{algpseudocode}

\begin{document}
\title{\textbf{On the distance-edge-monitoring numbers of graphs} \footnote{Supported by the National
Science Foundation of China (Nos. 12061059, 11601254, 11551001,
11161037, 61763041, 11661068, and 11461054), the Qinghai Key
Laboratory of Internet of Things Project (2017-ZJ-Y21) and
Science \& Technology development Fund of Tianjin Education Commission for Higher Education, China (2019KJ090).
} }

\author{Chengxu Yang
\footnote{School of Computer, Qinghai Normal
University, Xining, Qinghai 810008, China. {\tt
cxuyang@aliyun.com}},
 \ \ Ralf
Klasing, \footnote{Corresponding author: Universit\'{e} de Bordeaux, Bordeaux INP, CNRS, LaBRI, UMR 5800, Talence, France.
{\tt ralf.klasing@labri.fr}}
 \ \ Yaping
Mao, \footnote{Academy of Plateau Science and Sustainability, Xining,
Qinghai 810008, China.
{\tt maoyaping@ymail.com}}
\ \ Xingchao Deng
\footnote{School of Mathematical Science, Tianjin Normal University, Tianjin, 300387, China. {\tt dengyuqiu1980@126.com}}
}
\date{}
\maketitle

\begin{abstract}
Foucaud {\it et al.}~[{\it Discrete Appl.~Math.}~319 (2022), 424–438] recently introduced and initiated the study of a new graph-theoretic concept in the area of network monitoring.
For a set $M$ of vertices and an edge $e$ of a graph $G$, let $P(M, e)$ be the set of pairs $(x, y)$ with a vertex $x$ of
$M$ and a vertex $y$ of $V(G)$ such that $d_G(x, y)\neq  d_{G-e}(x, y)$. For a vertex $x$, let $EM(x)$ be the set of edges $e$ such that there exists a vertex $v$ in $G$ with $(x, v) \in  P(\{x\}, e)$. A set $M$ of vertices of a graph $G$ is \emph{distance-edge-monitoring set} if every edge $e$ of $G$ is monitored by some vertex of $M$,
that is, the set $P(M, e)$ is nonempty. The \emph{distance-edge-monitoring number} of a graph $G$,
denoted by $dem(G)$,
is defined as the smallest size of distance-edge-monitoring sets of $G$.
The vertices of $M$ represent distance probes in a network modeled by $G$; when the edge $e$ fails, the distance from $x$ to $y$ increases, and thus we are able to detect the failure. It turns out that not only we can detect it, but we can even correctly locate the failing edge.
In this paper, we continue the study of \emph{distance-edge-monitoring sets}. In particular,
we give upper and lower bounds of $P(M,e)$,
$EM(x)$, $dem(G)$, respectively, and extremal graphs attaining the bounds are characterized.
We also characterize the graphs with $dem(G)=3$.\\[2mm]
\end{abstract}

\section{Introduction}

Foucaud {\it et al.}~\cite{FKKMR21} recently introduced a new concept of network monitoring using distance probes, called \emph{distance-edge-monitoring}. Networks are naturally modeled by finite undirected simple connected graphs, whose vertices represent computers and whose edges represent connections between them. We wish to be able to monitor the network in the sense that when a connection (an edge) fails, we can detect this failure. We will select a (hopefully) small set of vertices of the network, that will be called \emph{probes}. At any given moment, a probe of the network can measure its graph distance to any other vertex of the network. The goal is that, whenever some edge of the network fails, one of the measured distances changes, and thus the probes are able to detect the failure of any edge.
Probes that measure distances in graphs are present in real-life networks, for instance this is useful in the fundamental task of \emph{routing}~\cite{DABVV06,GT00}. They are also frequently used for problems concerning \emph{network verification}~\cite{BBDGKP15,BEEHHMR06,BEMW10}.

We will now present the formal definition of the concept of \emph{distance-edge-monitoring sets}, as introduced by Foucaud {\it et al.}~\cite{FKKMR21}.
Graphs considered are finite, undirected and simple.
Let $G=(V,E)$ be a graph with vertex set $V$ and edge set $E$, respectively.
We denote by $d_G(x,y)$ the distance between two vertices $x$ and $y$ in a graph $G$.
For an edge $e$ of $G$, we denote by $G-e$ the graph obtained by deleting $e$ from $G$. 


\begin{definition}\label{Defination:$P(M, e)$}
For a set $M$ of vertices and an edge $e$ of a graph $G$, let $P(M, e)$ be the set of pairs $(x, y)$ with a vertex $x$ of
$M$ and a vertex $y$ of $V(G)$ such that $d_G(x, y)\neq  d_{G-e}(x, y)$. In other words, $e$ belongs to all shortest paths between $x$ and $y$
in $G$.
\end{definition}

\begin{definition}
For a vertex $x$, let $EM(x)$ be the set of edges $e$ such that there exists a vertex $v$ in $G$ with $(x, v) \in  P(\{x\}, e)$, that is
$EM(x)=\{e\,|\,e \in E(G) \textrm{~and~ }
\exists v \in V(G)\textrm{~such that~}
d_G(x,v)\neq d_{G-e}(x,v)\},$
or $EM(x)=\{e\,|\,e \in E(G) \textrm{and }
P(\{x\}, e) \neq \emptyset \}$.
If $e \in EM(x)$,
we say that \emph{$e$ is monitored by $x$}.
\end{definition}

\begin{definition}
A set $M$ of vertices of a graph $G$ is \emph{distance-edge-monitoring set} if every edge $e$ of $G$ is monitored by some vertex of $M$,
that is, the set $P(M, e)$ is nonempty. Equivalently, $\bigcup\limits_{x\in M}EM(x)=E(G)$.
\end{definition}

One may wonder about the existence of such an edge detection set $M$. The
answer is affirmative. If
we take $M=V(G)$, then
$$
E(G) \subseteq  \bigcup\limits_{x\in V(G)}N(x) \subseteq \bigcup\limits_{x\in V(G)}EM(x).
$$
Therefore, we consider the smallest cardinality of $M$ and give the following parameter.

\begin{definition}
The \emph{distance-edge-monitoring number} $dem(G)$ of a graph $G$ is defined as the smallest size of a \emph{distance-edge-monitoring set} of $G$, that is
$$
dem(G)=\min\left\{|M|\Big| \bigcup_{x\in M}EM(x)=E(G)\right\}.
$$
\end{definition}

The vertices
of $M$ represent distance probes in a network modeled by $G$, \emph{distance-edge-monitoring sets} are very effective in network fault tolerance testing. For example, a distance-edge-monitoring set can detect a failing edge, and
it can correctly locate the failing edge by distance from $x$ to $y$, because the
distance from $x$ to $y$ will increases when the edge $e$ fails.
Concepts related to {\it distance-edge-monitoring sets} have been considered e.g.~in
\cite{BBDGKP15,BBKS17,BEEHHMR06,BR06,HM76,mixedMD,edgeMD,geodetic,strong-resolving,ST04,S75}.
A detailed discussion of these concepts can be found in \cite{FKKMR21}.

Foucaud {\it et al.}~\cite{FKKMR21} introduced and initiated the study of distance-edge-monitoring sets.
They showed that
for a nontrivial connected graph $G$ of order $n$, $1\leq dem(G)\leq n-1$ with $dem(G)=1$ if and only if $G$ is a tree,
and $dem(G)=n-1$ if and only if it is a complete graph. They derived the exact value of $dem$ for grids, hypercubes, and complete bipartite graphs.
Then, they related $dem$ to other standard graph parameters. They showed that $dem(G)$ is lower-bounded by the arboricity of the graph, and upper-bounded by its vertex cover number. It is also upper-bounded by twice its feedback edge set number. Moreover, they characterized connected graphs $G$ with $dem(G)=2$.
Then, they showed that determining $dem(G)$ for an input graph $G$ is an NP-complete problem, even for apex graphs. There exists a polynomial-time logarithmic-factor approximation algorithm, however it is NP-hard to compute an asymptotically better approximation, even for bipartite graphs of small diameter and for bipartite subcubic graphs. For such instances, the problem is also unlikely to be fixed parameter tractable when parameterized by the solution size.

In this paper, we continue the study of \emph{distance-edge-monitoring sets}. In particular,
we give upper and lower bounds of $P(M,e)$,
$EM(x)$, $dem(G)$, respectively, and extremal graphs attaining the bounds are characterized.
We also characterize the graphs with $dem(G)=3$.

\section{Preliminaries}

Graphs considered are finite, undirected and simple.
Let $G=(V,E)$
be a graph with vertex set $V$
and edge set $E$, respectively.
The \emph{neighborhood
set} of a vertex $v\in V(G)$ is $N_G(v)=\{u\in V(G)\,|\,uv\in
E(G)\}$.
Let $N_G[v]=N_G(v)\cup \{v\}$.
The \emph{degree} of a vertex $v$ in $G$ is denoted by
$d(v)=|N_{G}(v)|$. $\delta(G)$, $\Delta(G)$ is the minimum,
maximum degree of the graph $G$, respectively.
For a vertex subset
$S\subseteq V(G)$, the subgraph induced by $S$ in $G$ is denoted by
$G[S]$ and similarly $G[V\setminus S]$ for $G\setminus S$ or $G-S$.
$v^{k+}$ is a vertex $v$ whose degree is at least $k$.
In a graph $G$, a vertex is a \emph{core vertex} if it is $v^{3+}$.
A path with all internal vertices of degree $2$ and whose end-vertices are \emph{core vertices} is called a \emph{core path} (note that we
allow the two end-vertices to be equal, but all other vertices must be distinct). A core path that is a cycle (that is, both
end-vertices are equal) is a \emph{core cycle}.
The \emph{base graph $G_b$} of a graph $G$ is the graph obtained from $G$ by iteratively removing
vertices of degree $1$. Clearly, $dem(G) = dem(G_b)$.

Foucaud {\it et al.}~\cite{FKKMR21} showed that
$1 \leq dem(G) \leq n-1$ for any $G$ with order $n$, and characterized graphs with $dem(G)=1,2,n-1$.

\begin{theorem}{\upshape\cite{FKKMR21}}\label{th-dem-1}
Let $G$ be a connected graph with at least one edge. Then $dem(G) = 1$ if and only if $G$ is a tree.
\end{theorem}

For two vertices $u,v$ of a graph $G$ and two non-negative integers $i,j$, we denote by $B_{i,j}(u, v)$ the set of vertices at
distance $i$ from $u$ and distance $j$ from $v$ in $G$.

\begin{theorem}{\upshape\cite{FKKMR21}}\label{th-dem-2}
Let $G$ be a connected graph with at least one cycle, and let $G_b$ be the base graph of $G$. Then, $dem(G) = 2$ if and
only if there are two vertices $u$, $v$ in $G_b$ such that all of the following conditions $(1)$-$(4)$ hold in $G_b$:

$(1)$ for all $i, j \in \{0, 1, 2, \cdots \}$, $B_{i,j}(u, v)$ is an independent set.

$(2)$ for all $i, j \in \{0, 1, 2, \cdots \}$, every vertex $x$ in
      $B_{i,j}(u, v)$ has at most one neighbor in each of the four sets
$B_{i-1,j}(u, v) \cup B_{i-1,j-1}(u, v)$, $B_{i-1,j}(u, v) \cup  B_{i-1,j+1}(u, v)$, $B_{i,j-1}(u, v) \cup $ $B_{i-1,j-1}(u, v)$ and $B_{i,j-1}(u, v) \cup B_{i+1,j-1}(u, v)$.

$(3)$ for all $i, j \in \{1, 2, \cdots \}$, there is no $4$-vertex path
      $zxyz'$ with $z \in  B_{i-1,a}(u, v)$, $z'\in B_{a',j}(u, v)$, $x \in  B_{i,j}(u, v)$, $y \in  B_{i-1,j+1}(u, v)$,
$a \in \{j-1, j + 1\}$, $a' \in \{i-2, i\}$.

$(4)$ for all $i, j \in \{1, 2, \cdots \}$, $x \in  B_{i,j}(u, v)$ has neighbors in at most two sets among
$B_{i-1,j+1}(u, v)$, $B_{i-1,j-1}(u, v)$, $B_{i+1,j-1}(u, v)$.
\end{theorem}

\begin{theorem}{\upshape\cite{FKKMR21}}
\label{th-dem-n}
$dem(G) = n-1$ if and only if $G$ is the complete graph of order $n$.
\end{theorem}

\section{Results for $P(M, e)$}

For the parameter $P(M, e)$, we have the following monotonicity property.
\begin{proposition}\label{Pro-P-1}
For two vertex sets $M_1,M_2$ and an edge $e$ of a graph $G$, if $M_1 \subset M_2$, then
$P(M_1, e) \subset P(M_2, e)$.
\end{proposition}
\begin{proof}
For any $(x, y)\in P(M_1, e)$ with $x\in M_1$ and $y\in V(G)$, we have $d_G(x, y)\neq  d_{G-e}(x, y)$. Since $M_1 \subset M_2$, it follows that $x \in M_2$. Since $d_G(x,y)\neq  d_{G-e}(x, y)$, we have $(x, y)\in P(M_2, e)$, and so $P(M_1, e) \subset P(M_2, e)$.
\end{proof}

From Proposition \ref{Pro-P-1}, one may think $P(M_1, e)\nsubseteq P(M_2, e)$ if $M_1 \nsubseteq M_2$.
\begin{proposition}\label{Pro-P-2}
For two vertex sets $M_1,M_2$ and an edge $e$ of a graph $G$, if $P(M_1 \cap M_2, e)\neq \emptyset$, then
$M_1 \cap M_2=\emptyset$
if and only if
$P(M_1, e) \cap  P(M_2, e)=\emptyset$.
\end{proposition}
\begin{proof}
If $M_1 \cap M_2=\emptyset$, then it follows from the definition of $P(M,e)$ that
$P(M_1, e) \cap  P(M_2, e)=\emptyset$.
Conversely, we suppose that $P(M_1, e) \cap  P(M_2, e)=\emptyset$.
Assume that $M_1 \cap M_2\neq \emptyset$. Let $M_1 \cap M_2=M$. Clearly,
$M \subset M_1$ and $M \subset M_2$.
From
Proposition \ref{Pro-P-1},
we have
$P(M, e) \subset P(M_1, e)$ and
$P(M, e) \subset P(M_2, e)$,
and hence
$P(M, e) \subseteq  P(M_1, e) \cap P(M_2, e)$.
Obviously,
$P(M_1, e) \cap P(M_2, e) \subseteq  P(M, e)$ and hence $P(M_1, e) \cap P(M_2, e)=P(M, e)$.
Since $P(M, e) \neq \emptyset$, it follows that
$P(M_1, e) \cap P(M_2, e) \neq \emptyset$,
a contradiction. So, we have $M_1 \cap M_2=\emptyset$.
\end{proof}

\subsection{Upper and lower bounds}

The following observation is immediate.
\begin{observation}{\upshape\cite{FKKMR21}}
Let $M$  be a distance-edge-monitoring set of a graph $G$. Then, for any two distinct edges $e_1$ and $e_2$ in $G$, we have
$P(M, e_1) \neq  P(M, e_2)$.
\end{observation}

For any graph $G$ with order $n$, if
$|M|=1$, then we have the following observation.
\begin{observation}
Let $G$ be a graph with order $n$, and $v \in V(G)$. Then
$$
0\leq |P(\{v\},uw)|\leq n-1.
$$
Moreover, the bounds are sharp.
\end{observation}

In terms of order of a graph $G$, we can derive the following upper and lower bounds.
\begin{proposition}\label{th-Bounds-P1-1}
Let $G$ be a graph of order $n$. For a vertex set $M$ and an edge $e$ of a graph $G$,
we have
$$
 0 \leq |P(M, e)| \leq n(n-1).
$$
Moreover, the bounds are sharp.
\end{proposition}
\begin{proof}
Clearly, $|P(M, e)|\geq 0$. From Proposition \ref{Pro-P-1}, we have $P(M, e) \subset P(V(G), e)$. Let $M=V(G)$. Then
the number of ordered pairs is $n(n-1)$ in $G$, and hence $|P(M, e)| \leq n(n-1)$,
as desired.
\end{proof}

To show the sharpness of the bounds in Proposition~\ref{th-Bounds-P1-1}, we consider the following examples.
\begin{example}
For any graph $H$, let $G=K_n \vee H$. Let $M=V(K_n)$ and $e\in E(H)$.
If $x,y\in M$, then $d_G(x,y)=d_{G-e}(x,y)=1$, and so
$(x,y) \notin P(M,e)$.
If $x\in V(K_n)$
and $y\in V(H)$, then $d_G(x,y)=d_{G-e}(x,y)=1$, and hence $(x,y) \notin P(M,e)$.
Clearly,
$P(M,e)=\emptyset$, and hence
$|P(M,e)|=0$. If $G=K_2$, then
$|P(M, e)| = n(n-1)$, which means that the bounds in Proposition \ref{th-Bounds-P1-1} are sharp.
\end{example}

The \emph{double star} $S(n,m)$ for integers $n\geq m\geq 0$ is the
graph obtained from the union of two stars $K_{1,n}$ and $K_{1,m}$
by adding the edge $e$ between their centers.
\begin{proposition}\label{th-Bounds-P2-2}
Let $G$ be a graph of order $n$
with a cut edge $e$. For any vertex set $M$, we have
$$
 2(n-1)\leq |P(M, e)| \leq 2\lfloor n/2\rfloor \lceil n/2 \rceil.
$$
Moreover, the bounds are sharp.
\end{proposition}
\begin{proof}
Let $G_1,G_2$ be the two components of $G\setminus e$, and let $|V(G_1)|=n_1$ and  $|V(G_2)|=n_2$. For any $x\in V(G_1)$ and
$y\in V(G_2)$, since $e$ is cut edge,
it follows that $d_G(x, y) \neq d_{G-e}(x, y)$.
If $M=V(G)$, then
$P(M, e)=\{(x,y),(y,x)| x\in V(G_1)~ and~y\in V(G_2) \}$, and hence
$|P(M, e)|=2|V(G_1)||V(G_2)|
=2n_1n_2=2n_1(n-n_1) \leq 2\lfloor
\frac{n}{2}\rfloor \lceil \frac{n}{2} \rceil$, and
so $|P(M, e)| \leq 2\lfloor
\frac{n}{2}\rfloor \lceil \frac{n}{2} \rceil$.
Since $|P(M, e)|=2n_1(n-n_1) \geq 2(n-1)$, it follows that $|P(M, e)| \geq 2(n-1)$.
\end{proof}

\begin{example}
Let $G$ be the double star $S(\lfloor n/2\rfloor-1,\lceil n/2 \rceil-1)$. If
$M=V(G)$, then $d_G(x,y)\neq d_{G-e}(x,y)$ for any $x\in V(K_{1,\lfloor n/2\rfloor-1})$ and $y\in V(K_{1,\lfloor n/2\rfloor-1})$. Then $(x,y), (y,x)\in P(M,e)$, and hence $|P(M,e)|\geq 2\lfloor n/2\rfloor \lceil n/2 \rceil$. From Proposition \ref{th-Bounds-P2-2}, we have $|P(M,e)|\leq 2\lfloor n/2\rfloor \lceil n/2 \rceil$ and hence $|P(M,e)|= 2\lfloor n/2\rfloor \lceil n/2 \rceil$.
\end{example}

In fact, we can characterize the graphs attaining the lower bounds in Proposition \ref{th-Bounds-P1-1}.
\begin{proposition}\label{Proposition:P(M,e)extremal value}
Let $G$ be a graph with
$uv \in E(G)$ and
$M \subset V(G)$. Then
$|P(M,uv)|=0$ if and only if one of the following conditions holds.
\begin{itemize}
\item[] $(i)$ $M = \emptyset$;

\item[] $(ii)$ $d_G(x,u)=d_G(x,v)$ for any $x \in M$.

\item[] $(iii)$ for any $x \in M$ and $d_G(x, u)=d_G(x, v)+1$, we have
$d_{G-uv}(x, u)=d_G(x, u)$.
\end{itemize}
\end{proposition}
\begin{proof}
Suppose that $|P(M,uv)|=0$.
Since
$$
P(M,uv)=\{(x, y)|d_G(x, y)\neq  d_{G-uv}(x, y),x \in M, y\in V(G)\}=\emptyset,
$$
it follows that $M=\emptyset$
or there exists a vertex set
$M\in V(G)$ and an edge $uv\in E(G)$
such that
$d_G(x, y)=d_{G-uv}(x, y)$ for any $x\in M$ and
$y\in V(G)$. For the fixed $x$, if $y=u$ and $y=v$, then we only need to consider the path
from $x$ to $y$ through $uv$, and hence $d_G(x, u)=d_{G-uv}(x, u)$
and $d_G(x, v)=d_{G-uv}(x, v)$.
Clearly, we have
$|d_G(x, v)-d_G(x, u)|\leq 1$. Without loss of generality, let
$d_G(x, u)\geq d_G(x, v)$.
For any $x\in M$, if
$d_G(x, u)=d_G(x, v)$, then $(ii)$ is true.

\begin{claim}
If $d_G(x, u)=d_G(x, v)+1$, then
$d_{G-uv}(x, u)=d_G(x, u)$.
\end{claim}
\begin{proof}
Assume, to the contrary, that
$d_{G-uv}(x, u) > d_G(x, u)$.
For $u \in V(G)$, we have $d_{G-uv}(x, u) \neq  d_G(x, u)$, and hence
$(x,u) \in  P(M, uv)=\emptyset$, a contradiction.
\end{proof}

Conversely, if $M=\emptyset$, then $|P(M,uv)|=0$. For any $x \in M$, suppose that $d_G(x,u)=d_G(x,v)$, then $d_G(x,y)=d_{G-uv}(x,y)$ for any $y \in V(G)$, and hence
$(x,y) \notin P(M,uv)$.
For any $x \in M$, if $d_G(x, u)=d_G(x, v)+1$ then
$d_{G-uv}(x, u)=d_G(x, u)$ and hence $d_G(x,y)=d_{G-uv}(x,y)$ for any $y \in V(G)$. It follows that $(x,y) \notin P(M,uv)$.
From the definition of $P(M, e)$,
we have $P(M, e)=\emptyset$,
and hence $|P(M, e)|=0$.
\end{proof}

In fact, we can characterize the graphs attaining the upper bounds in Proposition~\ref{th-Bounds-P2-2}.
\begin{proposition}\label{th-P-Upper}
Let $G$ be a graph with a cut edge $v_1v_2\in E(G)$ and $M=V(G)$. Then $|P(M,v_1v_2)|=2\lfloor \frac{n}{2}\rfloor \lceil \frac{n}{2} \rceil$ if and only if there are two vertex disjoint subgraphs $G_1$ and $G_2$ with $V(G)=V(G_1)\cup V(G_2)$ and $||V(G_1)|-|V(G_2)|| \leq 1$, where
$v_i\in V(G_i)$, $i=1,2$. In addition, $G_1$
and $G_2$ is connected by a bridge edge $v_1v_2$.
\end{proposition}
\begin{proof}
Suppose that $|P(M,v_1v_2)|=2\lfloor \frac{n}{2}\rfloor \lceil \frac{n}{2} \rceil$. Since $M=V(G)$, it follows that there are two induced subgraphs $G_1$ and $G_2$ with $V(G)=V(G_1)\cup V(G_2)$, where
$v_i\in V(G_i)$, $i=1,2$. Note that
$v_1v_2$ is a cut edge of $G$.
\begin{claim}\label{Claim-P-U}
If $x, y\in V(G_i)$, then
$(x,y) \notin P(M, e)$
and $(y, x) \notin P(M, e)$, where $i=1, 2$.
\end{claim}
\begin{proof}
Assume, to the contrary, that $x, y\in V(G_i)$
and $(x,y) \in P(M, e)$, where $i=1,2$.
Then there exists a shortest path
from $x$ to $y$ such that $d_G(x,y)\neq d_{G-v_1v_2}(x,y)$,
where $v_i \in V(G_i)$, $i=1,2$.
Since $v_1v_2$ is a cut edge, it follows that
$d_G(x, y)=d_{G-v_1v_2}(x, y)$, and hence $(x,y)\notin P(M,e)$, a contradiction.
\end{proof}

By Claim \ref{Claim-P-U}, we only consider that $x\in V(G_i)$ and $y\in V(G)-V(G_i)$ ($i=1, 2$).
Since $v_1v_2$ is a cut edge, it follows that
$d_G(x, y) \neq d_{G-v_1v_2}(x, y)$, and hence $(x,y)\in P(M,e)$. It follows that
$|P(M, e)|= 2|V(G_1)||V(G_2)|=2|V(G_1)|(n-|V(G_1)|)
\leq 2\lfloor \frac{n}{2}\rfloor \lceil \frac{n}{2} \rceil$,
where the equality holds if and only
$|V(G_1)|=\lfloor\frac{n}{2}\rfloor$
or $|V(G_1)|=\lceil\frac{n}{2}\rceil$, and hence $||V(G_1)|-|V(G_2)|| \leq 1$.

Conversely, we suppose that there are two vertex disjoint subgraphs $G_1$ and $G_2$ with $V(G)=V(G_1)\cup V(G_2)$ and $||V(G_1)|-|V(G_2)|| \leq 1$, where
$v_i\in V(G_i)$, $i=1,2$.
Then $G_1$ and $G_2$ are connected by a bridge edge, and hence
$|P(M, e)| = |P(V(G), e)|= 2|V(G_1)||V(G_2)|
=2\lfloor n/2\rfloor \lceil n/2 \rceil
$, as desired.
\end{proof}

For $|P(M,e)|$, we give some results for some special graphs.
\begin{lemma}\label{lem:COM}
Let $K_n$ be a complete graph, and let $M\subseteq V(K_n)$. Then
$$
 P(M,uv)=
\begin{cases}
\{(u,v),(v,u)\} & \mbox{if}~u,v\in M,\\
\{(u,v)\} & \mbox{if}~u\in M~\text{and}~v\notin M,\\
\{(v,u)\} & \mbox{if}~v\in M~\text{and}~u\notin M,\\
\emptyset,& \mbox{if}~u,v\notin M,
\end{cases}
$$
where $uv\in E(K_n)$.
\end{lemma}
\begin{proof}
Let
$V(K_n)=\{v_1,v_2,\cdots,v_n\}$. For any edge $uv$,
if $u \in M$ and $v \notin M$, then
$P(M, uv)=\{(x, y)| x \in M, y \in V(G)  \text{~and~} d_G(x, y)\neq  d_{G-uv}(x, y)\}$. Since
$d_{K_n}(u, v)=1$ and $d_{K_n-uv}(u, v)=2$,
we have
$(u, v) \in  P(M, xy)$.
The result follows for $u \in M$
and $v \notin M$.
Similarly, if $u,v \in M$,
then $P(M, e)=\{(u,v),(v,u)\}$.
Suppose that $u\notin M$ and $v\notin M$.
Let $P_{x,y}$ be the shortest path from $x\in M$
to $y\in V(G)$,
and hence there is no
the shortest path $P_{x,y}$ such that
$uv \notin E(P_{x,y})$,
and hence
$P(M, uv)= \emptyset$.
\end{proof}

The following corollary is immediate.
\begin{proposition}
Let $K_n$ be a complete graph, and let $M\subseteq V(K_n)$. Then
$$
0 \leq |P_G(M,uv)| \leq 2,
$$
where $uv\in E(K_n)$. Furthermore, $|P_G(M,uv)|=0$ if and only if $u,v\notin M$; $|P_G(M,uv)|=2$ if and only if $u,v\in M$; $|P_G(M,uv)|=1$ for otherwise.
\end{proposition}
\begin{proof}
For any $uv \in E(G)$ and $M\in V(G)$,
if $u,v \notin M$, then it follows from Lemma \ref{lem:COM} that
$P_G(M,uv)=\emptyset$, and
hence $|P_G(M,uv)|=0$.
If $u,v \in M$, then
it follows from Lemma \ref{lem:COM} that
$P_G(M,uv)=\{(u,v),(v,u)\}$, and so $|P_G(M,uv)|=2$.
Similarly, for other case, we have
$|P_G(M,uv)|=1$.
\end{proof}

\section{Results for $EM(x)$}
For $EM(x)$, we can observe some basic properties of distance-edge-monitoring sets.
Obviously, for any bridge edge $e \in E(G)$, the edge $e\in EM(x)$, which is given by Foucaud {\it et al.}~in \cite{FKKMR21}, see  Theorem \ref{Th-Bridge}.
\begin{theorem}{\upshape\cite{FKKMR21}}
\label{Th-Bridge}
Let $G$ be a connected graph and let $e$ be a bridge edge of $G$. For any vertex $x$ of $G$,
we have $e\in EM(x)$.
\end{theorem}

The following corollary is immediate.
\begin{corollary}
For a vertex $v$ of a tree $T$,  we have
$EM(v)=E(T)$.
\end{corollary}
\begin{proof}
For a vertex $v$ of a tree $T$, we have $EM(v)\subset E(T)$. Since any edge $e\in E(T)$ is a bridge edge of $T$, it follows from
Theorem \ref{Th-Bridge} that $e\in EM(v)$ for any vertex $v\in V(T)$, and
hence $E(T)\subset EM(v)$.
\end{proof}

\begin{theorem}{\upshape\cite{FKKMR21}}\label{Th-cover}
Let $G $ be a connected graph with a vertex $x$ of $G$. The following two conditions are equivalent:

$(1)$ $EM(x)$ is the set of edges incident with $x$.

$(2)$ For $y\in V(G)-N_G[x]$,
there exist two shortest paths from $x$ to $y$ sharing at most one edge.
\end{theorem}

Now, let's investigate the edges of  $EM(x)$ in $G$. Firstly, we introduced the following result, which is given in Foucaud {\it el al.}~\cite{FKKMR21}.
\begin{theorem}{\upshape\cite{FKKMR21}}
\label{Th-Ncover}
Let $G $ be a connected graph with a vertex $x$ of $G$ and for any $y\in N(x)$, then,
we have $xy \in EM(x)$.
\end{theorem}
By Theorem \ref{Th-Ncover}, we can obtain a lower bound on $EM(x)$ for any graph $G$ with minimum degree $\delta$, the description is as follows.
\begin{corollary}\label{cor-3-4}
Let $G$ be a connected graph. For any $x\in V(G)$, we have
$$
|EM(x)|\geq |N_G(x)|\geq \delta(G),
$$
with equality if and only if
$G$ is a regular graph such that
there exist two shortest paths from $u$ to $x$ sharing at most one edge, where $u\in V(G)-N_G[x]$. For example, a balanced complete bipartite graph $K_{n,n}$.
\end{corollary}

\begin{theorem}{\upshape\cite{FKKMR21}} \label{Th-forest}
For a vertex $x$ of a graph $G$, the set of edges $EM(x)$ induces a forest.
\end{theorem}

For a graph $G$ and a vertex $x \in V(G)$, one can derive the edge set $EM(x)$ from $G$ by Algorithm~\ref{algorithm:EM_v}. This algorithm is based on
the breadth-first spanning tree algorithm.
In the process of finding breadth-first spanning trees, we delete some edges that cannot be monitored by vertex $x$, and obtain the edge set $EM(x)$
when the algorithm terminates.
The time complexity of the breadth-first search tree algorithm is $O(|V(G)|+|E(G)|)$.
In Algorithm \ref{algorithm:EM_v}, we only add
the steps of deleting specific edges and checking neighbor vertex
shown in Lines 17--26.

\medskip
\begin{algorithm}[!htbp]
\small
\caption{The algorithm of finding an edge set $EM(x)$ in $G$}
\begin{multicols}{2}\label{algorithm:EM_v}
\begin{algorithmic}[1]
\Require a  graph $G$ and  a vertex $x\in V(G)$;
\Ensure A edge set  $EM(x)$ in $G$;
\For{ each vertx $u\in V(G)-\{x\}$}
\State colour[u] $\gets$ White
\State d[u] $\gets \infty$
\EndFor
\State $EM(x)\gets E(G)$
\State d[x] $\gets 0$
\State Q $\gets \emptyset$
\State  $\textbf{Enqueue}[Q,x]$
\While{$ Q \neq \emptyset $}
\State
$ u \gets \textbf{Dequeue}[Q]$
\State $ N'[u] \gets \emptyset $
\For{ each vertx $v \in Adj[u]$}
\If {colour[v] $\gets$ White}
\State $ N'[u]  \gets N'[u] \cup \{v\} $
\State colour[v] $\gets$ Gray
\State $ d[v] \gets d[u]+1$
\State Enqueue[Q,v]
\EndIf
\For {$v_i,v_j \in N'[u] $}
\If {$ v_iv_j\in E(G)$}\\
~~~~~~~~~~~~~~~~~$EM(x)=EM(x)-v_iv_j$
\EndIf
\EndFor
\State Dv $\gets \emptyset$
\For{each vertx $v_o \in Adj[v]$}
\If{$\textrm{colour}[v_o]=\textrm{ Gray}$}
\State $D_v \gets D_v  \cup \{v_o\} $
\EndIf
\EndFor
\If{$|D_v| \geq 1$ }
\For{$v_o\in D_v$}\\
~~~~~~~~~~~~~~~~~~~~$EM(x)=EM(x)-vv_o$
\EndFor
\EndIf
\EndFor
\State colour[u] $\gets$ DarkGary
\EndWhile
\State \Return$EM(x)$
\end{algorithmic}
\end{multicols}
\end{algorithm}

\medskip
We now give upper and lower bounds on $EM(x)$ in terms of the order $n$.
\begin{proposition}
Let $G$ be a connected graph with $|V(G)|\geq 2$. For any $v \in V(G)$, we have
$$
1\leq |EM(v)| \leq |V(G)|-1.
$$
Moreover, the bounds are sharp.
\end{proposition}
\begin{proof}
For any vertex $v \in V(G)$,
it follows from Theorem \ref{Th-forest} that
the set of edges $EM(x)$ induces a forest $F$
in $G$, and hence $|EM(v)|\leq |E(F)|\leq |E(T)|=|V(G)|-1$, where $T$ is a spanning tree of $G$.
Since $G$ is a connected graph, it follows from
Corollary \ref{cor-3-4} that
$|EM(x)| \geq \delta(G)\geq 1$,
and hence $|EM(v)| \geq 1$.
\end{proof}

Given a vertex $x$ of a graph $G$ and an integer $i$, let $N_i(x)$ denote the set of vertices at distance $i$ of $x$ in $G$.
Is there a way to quickly determine whether $e \in EM(v)$ or $e \notin EM(v)$?
Foucaud {\it et al.}~\cite{FKKMR21} gave the following characterization about edge $uv$ in $EM(x)$.

\begin{theorem}{\upshape\cite{FKKMR21}} \label{th-EM}
Let $x$ be a vertex of a connected graph $G$. Then,
$uv \in EM(x)$ if and only if
$u \in N_i(x)$ and $v$ is the
only neighbor of $u$ in $N_{i-1}(x)$,
for some integer $i$.
\end{theorem}

The following results are immediate from Theorem \ref{th-EM}. These results show that it is easy to determine $e\notin EM(v)$ for $v\in V(G)$.

\begin{corollary}\label{cor-EM}
Let $G$ be a connected graph, and $x\in V(G)$. Let $\mathcal{P}_{x,y}$ denote the set of
shortest paths from $x$ to $y$. Suppose that $uv$ is an edge of
$G_b$ satisfying one of the following conditions.

$(1)$ there exists an odd cycle $C_{2k+1}$ containing the vertices $x',u,v$ such
that $V(\mathcal{P}_{x,x'})\cap V(C_{2k+1})=x'$
and $d_G(x',u)=d_G(x',v)=k$.

$(2)$ there exists an even cycle $C_{2k}$ containing the vertices $x',u,v$ such that $V(\mathcal{P}_{x,x'})\cap V(C_{2k})=x'$, $d_G(x',u)=k-1$ and $d_G(x',v)=k$.

Then $uv \notin EM(x)$.
\end{corollary}
\begin{proof}
Since $d_G(x',u)=d_G(x',v)=k$, it follows that
$d_G(x',u)=d_{G-uv}(x',u)=k$
and
$d_G(x',v)=d_{G-uv}(x',v)=k$. Since $V(\mathcal{P}_{x,x'})\cap
 V(C_{2k+1})=x'$, it follows that
$d_G(x,u)=d_G(x,x')+d_G(x',u)$
and
$d_G(x,v)=d_G(x,x')+d_G(x',v)$,
and so
$d_G(x,u)=d_{G-uv}(x,u)$ and
$d_G(x,v)=d_{G-uv}(x,v)$. Clearly,
$uv\notin EM(x)$, and so $(1)$ holds.
From Theorem \ref{th-EM},
the results are immediate,
and hence $(2)$ holds.
\end{proof}

\begin{theorem}
For any $k \ (1\leq k\leq n-1)$, there exists a graph of order $n$ and a vertex $v\in V(G)$ such that $|EM(v)|=k$.
\end{theorem}
\begin{proof}
Let $F_1$ be a graph of order $k$ and $F_2$ be a graph obtained from $F_1$ by adding a new vertex $v$ and then adding all edges from $v$ to $V(F_1)$. Let $H$ be a graph obtained from $F_2$ and a graph $F_3$ of order $n-k-1$ such that there are at least two edges from each vertex in $F_3$ to $V(F_1)$.

From Corollary \ref{cor-3-4}, we have $|EM(v)|\geq |N_G(v)|=k$. To show $|EM(v)|\leq k$, it suffices to prove that $EM(v)=E_H[v,V(F_1)]$. Clearly, $E_H[v,V(F_1)]\subseteq EM(v)$. We need to prove that $EM(v)\subseteq E_H[v,V(F_1)]$, that is, $EM(v)\cap (E(H)\setminus E_H[v,V(F_1)])=\emptyset$. It suffices to show that for any $xy\in E(H)\setminus E_H[v,V(F_1)]$, we have $d_G(v,x)=d_{H-xy}(v,x)$ or $d_H(v,y)=d_{H-xy}(v,y)$.
Note that $E(H)\setminus E_H[v,V(F_1)]=E(F_1)\cup E(F_3)\cup E_{H}[V(F_1),V(F_3)]$. If $xy\in E(F_1)$, then $d_H(v,x)=d_{H}(v,y)=1$, and it follows from Corollary \ref{cor-EM} $(1)$ that $xy \notin EM(v)$. If $xy\in E(F_3)$, then $d_H(v,x)=d_{H}(v,y)=2$, and it follows from Corollary \ref{cor-EM} $(1)$ that $xy \notin EM(v)$. Suppose that $xy\in E_{H}[V(F_1),V(F_3)]$. Without loss of generality, let $x\in V(F_1)$ and $y\in V(F_3)$. Since there are at least two edges from $y$ to $V(F_1)$, it follows that there exists a vertex $z\in V(F_1)$ such that $zy\in E(H)$. Then $d_H(v,z)=d_{H}(v,x)=2$ and $d_{H}(v,y)=3$. From Corollary \ref{cor-EM} $(2)$, we have $xy \notin EM(v)$. From the above argument, $|EM(v)|\leq k$, and hence $|EM(v)|=k$.
\end{proof}

Graphs with small values of $|EM(v)|$ can be characterized in the following.
\begin{theorem}
For a connected graph $G$ and $v\in V(G)$, we have
$|EM(v)|=1$ if and only if $G=K_2$.
\end{theorem}
\begin{proof}
If $|EM(v)|=1$, then it follows from Corollary \ref{cor-3-4} that $d_G(v)\leq 1$. Since $G$ is connected, it follows that $d_G(v)\geq 1$ and hence $d_G(v)=1$. Let $u$ be the vertex such that $vu\in E(G)$.
\begin{claim}\label{claim-EM-1}
$d_G(u)=1$.
\end{claim}
\begin{proof}
Assume, to the contrary, that $d_G(u)\geq 2$. For any vertex $y\in N_G(u)-v$, we have
$y\in N_2(v)$, and hence $d_G(y,v)=2$, and so $N_1(v)=\{u\}$. From Theorem \ref{th-EM},
$uy\in EM(v)$, and hence $|EM(v)|\geq 2$, a contradiction.
\end{proof}
From Claim \ref{claim-EM-1}, we have $d_G(u)=1$. Since $G$ is connected, it follows that
$G=K_2$.

Conversely, let $G=K_2$. For any $v\in V(K_2)$,
we have $|EM(v)|=\{uv\}$,
and hence $|EM(v)|=1$.
\end{proof}

We now define a new graph $A_d \ (d\geq 3)$ such that the eccentricity of $v$ in $A_d$ is $d$ and all of the following conditions are true.
\begin{itemize}
\item[] For each $i \ (2\leq i \leq d)$, let $B_i$ be a graph such that $|B_i|\geq 2$ for $2\leq i \leq d-1$.

\item[] $V(A_d)=\{v,u_1,u_2\}\cup (\bigcup_{2\leq i \leq d}V(B_i))$, where $B_1$ is a graph with vertex set $\{u_1,u_2\}$.

\item[] $E(A_d)=\{vu_1,vu_2\}\cup (\bigcup_{2\leq i \leq d}E_{A_d}(B_i)\,\,)\cup (\bigcup_{2\leq i\leq d}E_{A_d}[v^{i},V(B_{i-1})]$ with $|E_{A_d}[v^{i},V(B_{i-1})]|\geq 2$, where $v^{i}\in V(B_i)$ for $2\leq i \leq d$.
\end{itemize}

Note that for each vertex in $B_i$, there are at least two edges from this vertex to $B_{i-1}$, where $2\leq i \leq d$.

For $d=2$, let $D$ be a graph of order $n-3$, $D_1(n)$ be a graph with $V(D_1(n))=\{v,u_1,u_2\}\cup V(D)$ and $E(D_1(n))=\{u_1w,u_2w\,|\,w\in V(D)\}\cup \{u_1v,u_2v,uv\}
\cup E(D)$, and $D_2(n)$ be a graph with $V(D_2(n))=\{v,u_1,u_2\}\cup V(D)$ and $E(D_2(n))=\{u_1w,u_2w\,|\,w\in V(D)\}\cup \{u_1v,u_2v\}\cup E(D)$.

\begin{theorem}
Let $G$ be connected graph with at least $3$ vertices. Then
there exists a vertex $v\in V(G)$ such that
$|EM(v)|=2$ if and only if $=D_1(n)$ or $G=D_2(n)$ or $G=A_d$ for $d\geq 3$.
\end{theorem}
\begin{proof}
Suppose that $G=D_1(n)$ or $G=D_2(n)$. Then there is a vertex $v\in V(G)$. Let $d$ be the eccentricity of $v$ in $G$.
For $w\in V(D)$, the subgraph induced by the vertices in $\{w,u_1,u_2,v\}$ is an even
cycle $C_4$, and hence $d_G(v,u_1)=1$ and $d_G(v, w)=2$. It follows from
Corollary \ref{cor-EM} that $wu_1\notin EM(v)$. Similarly, we have $wu_2\notin EM(v)$. If $u_1u_2\in E(G)$, then the subgraph induced by the vertices in
$\{u_1,u_2, v\}$ is a $3$-cycle, and hence $d_G(v,u_1)=1$
and $d_G(v, u_2)=1$. From Corollary \ref{cor-EM}, we have $u_1u_2\notin EM(v)$.
Similarly, we have
$d_G(v, w_i)=2$ and $d_G(v, w_j)=2$ for $w_iw_j \in E(D)$. From Corollary \ref{cor-EM},
we have $w_iw_j\notin EM(v)$, and hence
$|EM(v)|=\{u_1v, u_1v\}$, and so $|EM(v)|=2$.

Suppose that $G=A_d$, where $d \geq 3$. Note that $d$ is the eccentricity of $v$ in $G$.
Then
$$
E(A_d)=\{vu_1,vu_2\}\cup \left(\bigcup_{2\leq i \leq d}E_{A_d}(B_i)\,\,\right)\cup \left(\bigcup_{2\leq i\leq d}E_{A_d}[v^{i},V(B_{i-1})]\right)
$$
with $|E_{A_d}[v^{i},V(B_{i-1})]|\geq 2$, where $v^{i}\in V(B_i)$ for $2\leq i \leq d$.
Since $d_G(v, u_{is})=i$
and $d_G(v, u_{it})=i$
for any $u_{is}u_{it}\in E(B_i)$, it follows from
Corollary \ref{cor-EM} that $u_{is}u_{it} \notin EM(v)$.
Similarly,
let
$\mathcal{C}_i=E_{A_d}[v^{i},V(B_{i-1})]$ with $|E_{A_d}[v^{i},V(B_{i-1})]|\geq 2$, where $v^{i}\in V(B_i)$ for $2\leq i \leq d$.
If $ yx \in \mathcal{C}_i$, then
$x\in N_{i-1}(v)$, $y\in N_i(i)$
and there exists a vertex $x_1 \in N_{i-1}(v)$
such that $yx_1\in E(G)$.
From Corollary \ref{cor-EM},
we have $yx \notin EM(v)$, and so
$|EM(v)|=\{u_1v, u_1v\}$,
and hence $|EM(v)|=2$.

Conversely,
if $|EM(v)|=2$, then it follows from
Corollary \ref{cor-3-4} that $d_G(v) \leq 2$.
If $d_G(v) = 1$,
without loss of generality,
let $uv\in E(G)$ and $y\in N_G(u)$,
then $uy\in EM(v)$,
and hence $|N_G(u)-v|=1$, and so $G\cong P_3$,
and hence $G\cong B_2(3)$.
Suppose that $d_G(v) = 2$.
Without loss of generality, let
$N_G(v)=\{u_1, u_2\}$. Suppose that $n=3$. If $u_1u_2\notin E(G)$,
then $G=D_2(3)$.
If $u_1u_2 \in E(G)$,
then the subgraph induced by the vertices in
$\{v,u_1,u_2\}$ is a $3$-cycle,
and hence
$d_G(v,u_1)=d_G(v,u_2)$. From
Corollary \ref{cor-EM}, we have
$u_1u_2 \notin EM(v)$, and hence
$G=D_1(3)$.

Suppose that $n\geq 4$. Since
$|EM(v)|=2$, it follows that $\{vu_1, vu_2\} \subseteq EM(v)$, and hence $e\notin EM(v)$
for any $e\in E(G)-\{vu_1, vu_2\}$.
\begin{claim}\label{CLM_w_i}
For any $i\geq 2$,
$y\in N_i(v)$
and $x \in N_{i-1}(v)$,
if $yx \in E(G)$,
then there exists a vertex $x_1 \in N_{i-1}(v)$
with $yx_1 \in E(G)$.
\end{claim}
\begin{proof}
Assume, to the contrary,
that there exists no $x_1\in N_{i-1}$
such that $yx_1 \notin E(G)$.
Then $d_G(v,y)=i$ but
$d_{G-yx}(v,y) \geq i+1$,
and so $yx\in EM(v)$, and hence $|EM(v)|\geq 3$,
a contradiction.
\end{proof}

If $d=2$, then for any $w \in V(G)-\{v,u_1,u_2\}$, it follows from Claim \ref{CLM_w_i} that
if $w\in N_2(v)$ and
$wu_1\in E(G)$,
then $wu_2 \in E(G)$.
For any $w_{s},w_{t}\in N_2(v)$, we assume that $w_{s}w_{t}\in E(G)$. Since $d_G(v,w_{s})=2$ and
$d_G(v,w_{t})=2$, it follows from
Corollary \ref{cor-EM} that
$w_{s}w_{t}\notin EM(v)$, and hence
$G=B_1(n)$ or $G=B_2(n)$.

If $d \geq 3$, then
$$
V(G^{d*})=\{v, u_1,u_2\}\cup \{u_{ij}\,|\,2\leq i\leq d,1\leq j\leq t_d \}
=\{v, u_1, u_2\}\cup
\{u_{21},\ldots,u_{2t_2}\}
\cup \cdots \cup
\{u_{d1},\ldots,u_{dt_d}\},
$$
where
$v\in N_0(v)$,
$u_1,u_2\in N_1(v)$,
$u_{21},\ldots u_{2t_2}\in N_2(v)$,
$\ldots$
$u_{d1},\ldots u_{dt_d}\in N_d(v)$
and $\sum_{i=2}^{i=d}{t_s}=n-3$.

By Claim \ref{CLM_w_i},
if $y\in N_i(v)$ and
$yx\in E(G)$,
then there exists a vertex $x_1\in N_{i-1}(v)$
and $x_1\neq x$ such that $yx_1 \in E(G)$, and hence
$yx \in E_{A_d}[v^{i},V(B_{i-1})]$ with $|E_{A_d}[v^{i},V(B_{i-1})]|\geq 2$, where $v^{i}\in V(B_i)$ for $2\leq i \leq d$.

For any $u_{is},u_{it}\in B_i(v)$
and $u_{is}u_{it}\in E(B_i)$,
since $d_G(v,u_{is})=i$ and
$d_G(v,u_{it})=i$, it follows from
Corollary \ref{cor-EM} that
$u_{is}u_{it}\notin EM(v)$,
and hence $B_i$ is a
graph with order at least $2$, and so
$$
E(A_d)=\{vu_1,vu_2\}\cup \left(\bigcup_{2\leq i \leq d}E_{A_d}(B_i)\,\,\right)\cup \left(\bigcup_{2\leq i\leq d}E_{A_d}[v^{i},V(B_{i-1})]\right)
$$
with $|E_{A_d}[v^{i},V(B_{i-1})]|\geq 2$, where $v^{i}\in V(B_i)$ for $2\leq i \leq d$.
Therefore, $G=A_d$.
\end{proof}

\begin{theorem}
Let $G$ be a connected graph of order $n$. Then there exists a vertex $v\in V(G)$ such that
$|EM(v)|=n-1$ if and only if for any $w\in V(G)$,
there are no $w_1,w_2\in N_G(w)$ such that $d_G(w_1,v)=d_G(w_2,v)=d_G(w,v)-1$.
\end{theorem}
\begin{proof}
Suppose that $|EM(v)|=n-1$. Since $G$
is a connected graph of order $n$, it follows from
Theorem \ref{Th-forest} that
$EM(v)$ forms a spanning tree of $G$.
\begin{claim}\label{Clm:chEMn-1}
For any vertex $w\in V(G)$, there exists a vertex $w_i\in N_{d_G(v,w)-1}(v)$ with
$w_iw\in EM(v)$.
\end{claim}
\begin{proof}
Assume, to the contrary, that there is no $w_i \in N_{d_G(v,w)-1}(v)$
with $w_iw\in EM(v)$.
it follows that $EM(v)$ is disconnected,
which contradicts to the fact that the subgraph induced by the edges in $EM(v)$ is connected.
\end{proof}

By Claim \ref{Clm:chEMn-1}, for any vertex $w\in V(G)$, there exists a vertex $w_i\in N_{d_G(v,w)-1}(v)$ with $w_iw\in EM(v)$. From Theorem \ref{th-EM},
$w_i$ is the unique neighbor of $w$ in $N_{d_G(v,w)-1}(v)$,
and hence for any $w\in V(G)$,
there are no two vertices $w_1,w_2\in N_G(w)$ such that $d_G(w_1,v)=d_G(w_2,v)=d_G(w,v)-1$.

Conversely,
we suppose that for any $w\in V(G)$,
there are no $w_1,w_2\in N_G(w)$ such that $d_G(w_1,v)=d_G(w_2,v)=d_G(w,v)-1$.
Since $G$ is connected,
it follows that there is only one vertex
$w_i\in N_{d_G(v,w)-1}(v)$.
From Theorem
\ref{th-EM}, we have $w_iw\in EM(v)$,
and hence $ |EM(v)|=n-1$.
\end{proof}

The existence of $dem(G)$ is obvious,
because $V(G)$ is always a distance-edge-monitoring set.
Thus, the definition of $dem(G)$ is  meaningful.
The \emph{arboricity} $arb(G)$ of a graph $G$ is the smallest number of sets into which $E(G)$ can be partitioned and such that each
set induces a forest. The \emph{clique number} $\omega(G)$ of $G$ is the order of a largest clique in $G$.

\begin{theorem}{\upshape\cite{FKKMR21}}\label{Theorem:Lowerbond}
For any graph $G$ of order $n$ and size $m$, we have $dem(G)\geq arb(G)$, and thus $dem(G) \geq \frac{m}{n-1}$ and $dem(G) \geq \frac{\omega(G)}{2}$.
\end{theorem}
We next see that \emph{distance-edge-monitoring sets} are relaxations of vertex covers. A  vertex set $M$  is called a \emph{vertex cover}
of $G$ if every edge of $G$ has one of its endpoints in $M$. The minimum cardinality of a vertex cover $M$ in $G$ is the \emph{vertex covering number} of $G$, denoted by $\beta(G)$.

\begin{theorem}{\upshape\cite{FKKMR21}}
\label{Theorem:Upperbond}
In any graph $G$ of order $n$, any vertex cover of $G$ is a distance-edge-monitoring set, and thus $dem(G) \leq \beta(G)$.
\end{theorem}

An \emph{independent set} is a set of vertices of $G$ such that no two vertices are adjacent.
The largest cardinality of an \emph{independent set} is the \emph{independence number} of $G$, denoted by $\alpha(G)$.

The following well-known theorem 
was introduced by Galla\'{i} in 1959.

\begin{theorem}[Galla\'{i} Theorem]{\upshape\cite{Char15}}
\label{The-Gal}
In any graph $G$ of order $n$, we have
$$
\beta(G)+\alpha(G)=n.
$$
\end{theorem}

\begin{corollary}
For a graph $G$ with order $n$, we have
$$
dem(G)\leq  n- \alpha(G).
$$
Moreover, the bound is sharp.
\end{corollary}
\begin{proof}
From Theorem \ref{The-Gal}, we have
$\beta(G)=n-\alpha(G)$.
From Theorem \ref{Theorem:Upperbond},
we have $dem(G) \leq \beta(G)$,
and hence $dem(G) \leq n-\alpha(G)$,
as desired.
For a complete graph $G=K_n$ or
complete bipartite graph $G=K_{m,n}$,
we have $dem(G)=n-\alpha(G)$.
\end{proof}

\begin{theorem}{\upshape\cite{FKKMR21}}\label{Theorem:JOINT_Operation}
For any graph $G$, we have
$\beta(G) \leq dem(G \vee  K_1) \leq \beta(G)+1$. Moreover, if $G$ has radius at least $4$,
then $\beta(G) = dem(G \vee  K_1)$.
\end{theorem}

Similarly to the proof of Theorem~\ref{Theorem:JOINT_Operation}, we can obtain the following result.
\begin{corollary}
For any graph $G$ and integer $m$, we have
$$
\beta(G) \leq dem(G \vee  mK_1) \leq \beta(G)+m.
$$
Moreover, the bounds are sharp.
\end{corollary}
\begin{proof}
For any graph $G$ and integer $m$, we have $dem(G \vee  mK_1) \leq \beta(G \vee  mK_1)$
by Theorem \ref{Theorem:Upperbond}. Clearly, $\beta(G \vee  mK_1)\leq \beta(G)+m$, and hence $dem(G \vee  mK_1) \leq \beta(G)+m$.
It suffices to show that an edge monitoring
set $M$ of $G \vee  mK_1$ also is cover set
of $G$. Without loss of generality,
suppose that $V(mK_1)=\{w_1,\cdots,w_m\}$.
If there exists an edge $uv\in E(G)$ with $u,v \notin M$, then
$uv$ is monitored by
$M \cap V(G)$ in $G \vee  mK_1$.
For any $x\in M$, we have
$ d_G(x,u)\in \{1, 2\}$. Similarly, $d_G(x,v)\in \{1, 2\}$.
By Corollary \ref{cor-EM},
we have $d_G(x,v) \neq d_G(x,u)$. Without loss of generality, let
$d_G(x,v)=1$ and $d_G(x,u)=2$, and hence
$xw_iv$ is a shortest path from
$x$ to $v$. From Corollary \ref{cor-EM},
$uv$ is not monitored by $M$, a contraction.
Then $x\in M$ or $y\in M$,
and hence $\beta(G) \leq  dem(G \vee  mK_1)$.
By Theorem \ref{Theorem:JOINT_Operation},
if $G$ has radius at least $4$ and $m=1$,
then $\beta(G) = dem(G \vee  K_1)$.
If $m=1$ and $G=K_n$, then
$dem(K_n \vee  K_1) = \beta(K_n)+1=n$, and hence the bound is sharp.
\end{proof}

\begin{proposition}
For any $r$-regular graph $G$ of order $n\geq 5$, we have
$$
\frac{rn}{2n-2} \leq dem(G)\leq n-1.
$$
Moreover, the bounds are sharp.
\end{proposition}
\begin{proof}
For any $r$-regular graph graph $G$ of order $n$, since $e(G)=\frac{rn}{2}$, it
follows from Theorem \ref{Theorem:Lowerbond} that
$dem(G) \geq \frac{m}{n-1}$, and hence $ dem(G) \geq \frac{rn}{2n-2}$.
From Theorem \ref{Theorem:Upperbond},
we have $dem(G)\leq n-1$.
From Theorem \ref{th-dem-n}, if $r=1$ and $n=2$, then
$dem(K_2)=1$, and hence the lower bound is tight.
\end{proof}

\section{Graphs with distance-edge-monitoring number three}

For three vertices $u,v,w$ of a graph $G$ and non-negative integers $i,j,k$, let $B_{i,j,k}$ be the set of vertices at distance $i$ from $u$ and distance $j$ from $v$
and distance $k$ from $w$ in $G$, respectively.
\begin{lemma}\label{lem-char}
Let $G$ be a graph with $u,v,w \in V(G)$, and
$i,j,k$ be three non-negative integers such
that $B_{i,j,k} \neq \emptyset$. If $x \in B_{i,j,k} $, $xy\in E(G)$, and
$$
T=\left\{(i',j',k')\,|\,i' \in \{i-1,i,i+1\},
j' \in \{j-1,j,j+1\},
k' \in \{k-1,k,k+1\}\right\},
$$
then
$y \in B_{i',j',k'} $, where
$(i',j',k') \in T$.
\end{lemma}
\begin{proof}
Since $x \in B_{i,j,k} $ and $xy\in E(G)$, it follows that $d_G(x,u)=i$, $d_G(x,v)=j$
and $d_G(x,w)=k$. We have the following claim.

\begin{claim}\label{claim-dem-lem}
$d_G(y,u)\in \{i-1,i, i+1\}$.
\end{claim}
\begin{proof}
Assume, to the contrary, that $d_G(y,u)\leq i-2$ or $d_G(y,u)\geq i+2$. If $d_G(y,u)\leq i-2$, then $d_G(x,u)\leq d_G(u,y)+d_G(y,x)=d_G(u,y)+1\leq i-1$, which contradicts to the fact that $d_G(x,u)=i$. If $d_G(y,u)\geq i+2$, then $i+2\leq d_G(y,u)\leq d_G(u,x)+d_G(x,y)=i+1$, a contradiction.
\end{proof}

From Claim \ref{claim-dem-lem}, we have $d_G(y,u)\in \{i-1,i, i+1\}$. Similarly, $d_G(y,v)\in \{j-1,j, j+1\}$ and $d_G(y,w)\in \{k-1,k, k+1\}$.
\end{proof}

\begin{theorem}\label{th-dem-3}
For a graph $G$, $dem(G)=3$ if and only if
there exists three vertices $u,v,w$ in $G_b$ such that all of the following conditions $(1)$-$(8)$ hold in $G_b$:
\begin{description}
\item[] $(1)$ For any $i,j,k\in \{0,1,2,\ldots,diam(G)\}$, $B_{i,j,k}$ is an independent set.

\item[] $(2)$ For any $i,j,k\in \{0,1,2,\ldots,diam(G)\}$ and any $xy,xy'\in E(G_b)$, if $y\in V(B_{i',j',k'})$, then $y'\not \in V(B_{i',j',k'})$, where $i'\in \{i-1, i\}$, $j'\in \{j-1, j\}$, and $k'\in \{k-1, k\}$.

\item[] $(3)$ For any $i,j,k\in \{0,1,2,\ldots,diam(G)\}$ and any $xy,xy'\in E(B_{i,j,k})$, if $y \in B_{i_1,j_1,k_1}$, then $y'\notin B_{i_2,j_2,k_2}$,
where $(i_1,j_1,k_1)$ and $(i_2,j_2,k_2)$ satisfy all the following conditions:
\begin{itemize}
\item[] $(3.1)$ if $(i_1,j_1,k_1)=(i,j-1,k)$,
then
$(i_2,j_2,k_2) \notin \{(i_2,j-1,k_2)\,|\,i_2\in \{i-1, i, i+1\},~k_2\in \{k-1,k,k+1\}\}$.

\item[] $(3.2)$ if
$(i_1,j_1,k_1) =(i-1,j-1,k-1)$,
then
$(i_2,j_2,k_2) \notin \{(i_2,j_2,k_2)\,|
\, i_2\in \{i-1,i\},
\, j_2\in \{j-1,j\},
\, k_2\in \{k-1,k\}\}$.

\item[] $(3.3)$
if $(i_1,j_1,k_1)=(i-1,j+1,k-1)$,
then
$(i_2,j_2,k_2)
\notin
\{
(i-1, j, k-1),
(i-1, j, k),
(i-1, j, k-1),
(i-1, j, k-1),
(i, j, k-1)
\}$.

\item[] $(3.4)$ if $(i_1,j_1,k_1)=(i,j-1,k-1)$,
then
$(i_2,j_2,k_2)\notin
\{
(i-1, j-1, k-1 ),
(i, j-1, k-1),
(i, j, k-1),
(i, j-1, k),
(i+1, j-1, k-1)\}$.

\item[] $(3.5)$
if $(i_1,j_1,k_1) = (i,j-1,k+1)$,
then
$(i_2,j_2,k_2) \notin
\{(i, j-1, k),
(i, j-1, k)\}$.
\end{itemize}

\item[] $(4)$ For any $i,j,k\in \{0,1,2,\ldots,diam(G)\}$, there is no $4$-path satisfying the following conditions.

\begin{itemize}
\item[] $(4.1)$ $z_1xyz_2$ is the $4$-path
      with $x \in  B_{i,j,k} $,
and $y\in  B_{i-1,j+1,k+1}$,
$z_1 \in  B_{i-1,a, b} $, and
$z_2\in B_{c,j, k}$, where $a \in \{j-1, j + 1\}$, $b \in \{k-1, k + 1\}$, $c\in \{i-2, i\}$.

\item[] $(4.2)$ $4$-vertex path
      $z_1xyz_2$ with
    $z_1 \in  B_{i-1,a, k-1}$,
    $z_2 \in B_{c,j, b}$,
    $x   \in  B_{i,j,k}$, and
    $y   \in  B_{i-1,j+1,k+1}$, where
$a \in \{j-1, j + 1\}$,
$b \in \{k-2, k \}$,
$c\in \{i-2, i\}$.

\item[] $(4.3)$ $4$-vertex path
      $z_2xyz_3$ with $x=B_{i,j,k}$, $y=B_{i,j-1,k+1}$ and
\begin{align*}
z_2 \in &
B_{i-1,j-1,k-1}
\cup B_{i-1,j-1,k}
\cup B_{i-1,j-1,k+1}
\cup B_{i,j-1,k-1}
\cup B_{i,j-1,k+1}
\cup B_{i+1,j-1,k-1}
\cup B_{i+1,j-1,k}\\
& \cup B_{i+1,j-1,k+1},
\\
z_3 \in &
B_{i-1,j-2,k}
\cup B_{i-1,j-1,k}
\cup B_{i-1,j,k}
\cup B_{i,j-2,k}
\cup B_{i,j,k}
\cup B_{i+1,j-2,k}\\
& \cup B_{i+1,j-1,k}
\cup B_{i+1,j,k}.
\end{align*}

\end{itemize}

\item[] $(5)$ For any $i,j,k\in \{0,1,2,\ldots,diam(G)\}$ and any $x \in  B_{i,j,k} $, $x$ has at most two neighbors in two of $B_{i-1,j-1,k-1} ,B_{i+1,j-1,k-1}(u, v,w),B_{i-1,j+1,k'} $,
      where $k'\in \{k-1,k,k+1\}$.

\item[] $(6)$ For any $i,j,k\in \{0,1,2,\ldots,diam(G)\}$ and any $x \in  B_{i,j,k} $,  there is no $4$-star $K_{1,4}$ with edge set
$E(K_{1,4})= \{yx, z_1x, z_2x,z_3x\}$
such that $y\in B_{i-1,j-1,k-1}$,
\begin{align*}
z_1 \in &
B_{i-1,j-1,k+1}
\cup B_{i-1,j,k+1}
\cup B_{i-1,j+1,k-1}\cup B_{i-1,j+1,k}
\cup B_{i-1,j+1,k+1}, \\
z_2 \in&
B_{i-1,j-1,k+1}
\cup B_{i,j-1,k+1}
\cup B_{i+1,j-1,k-1}
\cup B_{i+1,j-1,k}
\cup B_{i+1,j-1,k+1}, \\
z_3  \in&
B_{i-1,j+1,k-1}
\cup B_{i,j+1,k-1}
\cup B_{i+1,j-1,k-1}
\cup B_{i+1,j,k-1}
\cup B_{i+1,j+1,k-1},
\end{align*}

\item[] $(7)$  There is a no $P_4^{+}$ satisfying the following conditions:

$(7.1)$ $V(P_4^{+})=\{z_1, z_2, x, y, z_3\}$
and $E(P_4^{+})=\{z_1x, z_3x, xy, yz_2\}$
such that $x=B_{i,j,k}$,
$y=B_{i-1,j+1,k-1}$, and
\begin{align*}
z_1\in&
B_{i-1,j-1,k-1}
\cup B_{i-1,j-1,k}
\cup B_{i-1,j-1,k+1}
\cup B_{i-1,j,k+1}
\cup B_{i-1,j+1,k-1}
\cup B_{i-1,j+1,k}
\cup B_{i-1,j+1,k+1},\\
z_2 \in&
B_{i-2,j,k-2}
\cup B_{i-2,j,k-1}
\cup B_{i-2,j,k}
\cup B_{i-1,j,k-2}
\cup B_{i-1,j,k}
\cup B_{i,j,k-2}
\cup B_{i,j,k-1}
\cup B_{i,j,k},\\
z_3\in&
B_{i-1,j-1,k-1}
\cup B_{i-1,j+1,k-1}
\cup B_{i,j-1,k-1}
\cup B_{i,j+1,k-1}
\cup B_{i+1,j-1,k-1}
\cup B_{i+1,j,k-1}
\cup B_{i+1,j+1,k-1}.
\end{align*}

$(7.2)$
$V(P_4^{+})=\{z_2, z_3, x, y, z_1\}$
and $E(P_4^{+})=\{z_2x, z_3x, xy, yz_1\}$
such that $x=B_{i,j,k}$,
$y=B_{i+1,j-1,k-1}$, and
\begin{align*}
z_1 \in&
B_{i,j-2,k-2}
\cup B_{i,j-2,k-1}
\cup B_{i,j-2,k}\cup B_{i,j-1,k-2}
\cup B_{i,j-1,k}\cup B_{i,j,k-2}\cup B_{i,j,k-1},\\
z_2\in&
B_{i-1,j-1,k-1}
\cup B_{i-1,j-1,k}
\cup B_{i-1,j-1,k+1}\cup B_{i,j-1,k-1}
\cup B_{i,j-1,k}
\cup B_{i+1,j-1,k+1} \cup B_{i+1,j-1,k-1}
\\
&\cup B_{i+1,j-1,k}
\cup B_{i+1,j-1,k+1},\\
z_3 \in&
B_{i-1,j-1,k-1}
\cup B_{i-1,j,k-1}
\cup B_{i-2,j+1,k-1} \cup B_{i,j-1,k-1}
\cup B_{i,j,k-1}
\cup B_{i,j+1,k-1}\cup B_{i+1,j-1,k-1} \\
&
\cup B_{i+1,j,k-1}
\cup B_{i+1,j+1,k-1}.
\end{align*}

\item[] $(8)$ There is no $3$-star $K_{1,3}$ with edge set $E(K_{1,3})=\{ xy, xz_1, x z_2\}$
such that $x=B_{i,j,k}$,
$y=B_{i,j-1,k-1}$, and
\begin{align*}
z_2 \in &
B_{i-1,j-1,k}
\cup  B_{i-1,j-1,k+1}
\cup  B_{i,j-1,k+1}
\cup  B_{i+1,j-1,k}
\cup  B_{i+1,j-1,k+1},\\
z_3 \in &
B_{i-1,j,k-1}
\cup B_{i-1,j+1,k-1}
\cup  B_{i,j+1,k-1}
\cup  B_{i+1,j,k-1}
\cup  B_{i+1,j+1,k-1},
\end{align*}

\end{description}
\end{theorem}
\begin{proof}
Assume that $dem(G) = 3$. Then $dem(G_b) = 3$.
Let $\{u, v,w\}$ be a distance-edgemonitoring set of $G_b$.

\setcounter{claim}{0}
\begin{claim}\label{claim1}
$B_{i,j, k} $ is an independent set.
\end{claim}
\begin{proof}
Assume, to the contrary, that $B_{i,j, k} $ is not an independent set.
Let $x,y\in B_{i,j, k}(u,v,w)$. Then $d_G(x,u)=d_G(y,u)=i$, $d_G(x,v)=d_G(y,v)=j$, and $d_G(x,w)=d_G(y,w)=k$, and hence $xy$ can not be monitored by $u,v,w$
by Theorem \ref{th-EM}, a contradiction.
\end{proof}

From Claim \ref{claim1}, $(1)$ holds.

For any $i,j,k\in \{0,1,2,\ldots,diam(G)\}$ and any $x\in B_{i,j,k} $ with $xy,xy'\in E(G_b)$, we assume that
$y\in B_{i',j',k'} $
and $(i',j',k')\neq (i,j,k)$.
Then we have the following claim.
\begin{claim}\label{claim2}
$y'\not \in B_{i',j',k'}$ for $i'\in \{i-1, i\}$, $j'\in \{j-1, j\}$, and $k'\in \{k-1, k\}$.
\end{claim}
\begin{proof}
Assume, to the contrary,
that $y'\in B_{i',j',k'} $.
We first suppose that
$y'\in B_{i-1 , j , k } $
(The case that $y, y' \in B_{i , j-1 , k } $ or $B_{i , j, k-1 } $ is symmetric). Then $d_G(y,u)=d_G(y',u)=i-1$. From Theorem \ref{th-EM},
$xy$ can not be monitored by $u$.
Since $x \in B_{i,j,k}(u, v,w)$ and
$y \in B_{i-1 , j , k } $, it follows that
$d_G(y,v)=d_G(x,v)=j$. From Theorem \ref{th-EM},
$xy$ can not be monitored by $v$. Similarly, since $d_G(y,w)=d_G(x,w)=k$,
it follows that $xy$ can not be monitored by $w$
by Theorem \ref{th-EM}, a contradiction.

Next, we suppose that $y, y' \in B_{i-1 , j-1 , k } $
(The case that $y, y' \in B_{i , j-1 , k-1 } $
or $B_{i-1 , j, k-1 } $ is symmetric).
Since $d_G(y,u)=d_G(y',u)=i-1$, it follows from Theorem \ref{th-EM} that $xy$ can not be monitored by $u$.
Similarly, $d_G(y,v)=d_G(y',v)=j-1$, $xy$ can not be monitored by $v$.
in addition,
$d_G(y, w)=d_G(x,w)=k$,
and hence
$xy$ is not monitored by $w$,
according to Theorem \ref{th-EM}.
So, $xy$ is not monitored by $u$, $v$, $w$,
a contradiction.

Finally, if $y, y' \in B_{i-1 , j-1 , k-1} $,
it follows that
$d_G(y, u)=d_G(y',u)=i-1$,
$d_G(y, v)=d_G(y',v)=j-1$,
similarly,
$d_G(y,w)=d_G(y',w)=k-1$,
by Theorem \ref{th-EM},
$xy$ is not monitored by $u$, $v$, $w$,
a contradiction.
\end{proof}
From Claim \ref{claim2}, $(2)$ holds.

For any $i,j,k\in \{0,1,2,\ldots,diam(G)\}$ and any $xy,xy'\in E(G_b)$, we suppose that $y\in B_{i,j-1,k}$. Then we have the following claim.
\begin{claim}\label{claim3}
$y'\notin B_{i_2,j_2,k_2}$ for
$(i_2,j_2,k_2) \in \{(i_2,j-1,k_2)\,|\,i_2\in \{i-1, i, i+1\},~k_2\in \{k-1,k,k+1\}\}.
$
\end{claim}
\begin{proof}
Assume, to the contrary, that $y' \in B_{i_2,j_2,k_2}$. Since $x \in B_{i,j,k}$
and $y,y'$ are both neighbors of $x$ and $y \in B_{i,j-1,k}$,
it follows that
$d_G(y, u)=d_G(x,u)=i$ and
$d_G(y, w)=d_G(x,w)=k$. From Theorem \ref{th-EM},
$xy$ is not monitored by $u$ and $w$.
Since $y' \in B_{i_2,j-1,k_2}$
and $y \in B_{i,j-1,k} $,
it follows that $d_G(y', v)=d_G(y,v)=j-1$, and hence
$xy$ is not monitored by $v$,
a contradiction.
\end{proof}

By Claim \ref{claim3}, $(3.1)$ holds. By the same method, we can prove that
$(3.2)$-$(3.6)$ all hold.

\begin{claim}\label{claim4}
For any $i,j,k\in \{0,1,2,\ldots,diam(G)\}$, there is no $4$-vertex path
$z_1xyz_2$ such that $x \in  B_{i,j,k} $, $y \in  B_{i-1,j+1,k+1} $,
$z_1 \in  B_{i-1,a, b} $, and $z_2\in B_{c,j, k} $, where $a \in \{j-1, j + 1\}$, $b\in \{k-1, k + 1\}$, $c\in \{i-2, i\}$.
\end{claim}
\begin{proof}
Assume, to the contrary, that there is a $4$-path satisfying the conditions of this claim.
Then $d_G(y,u)=d_G(z_1,u)=i-1$, $d_G(x,v)=d_G(z_2,v)=j$ and $d_G(x, w)=d_G(z_2,w)=k$, and from Theorem \ref{th-EM}, $xy$ can not be monitored by $u,v,w$, respectively, a contradiction.
\end{proof}

By Claim \ref{claim4}, $(4.1)$ holds.
Similarly, the conditions $(4.2)$ and $(4.3)$ can be easily proved.

\begin{claim}\label{claim5}
For any $i,j,k\in \{0,1,2,\ldots,diam(G)\}$ and any $x\in  B_{i,j,k}$, $x$ has at most two neighbors in two of $B_{i-1,j-1,k-1},B_{i+1,j-1,k-1}(u, v,w),B_{i-1,j+1,k'}$,
where $k'\in \{k-1,k,k+1\}$.
\end{claim}
\begin{proof}
Assume, to the contrary, that
$x\in B_{i,j,k}$ has three neighbors $y,y',y''$
such that
$y \in B_{i-1,j-1,k-1} $,
$y' \in B_{i+1,j-1,k-1}(u, v,w)$,
$y'' \in B_{i-1,j',k'} $
where $k'\in \{k-1,k,k+1\}$
and $j'\in \{j-1, j, j+1\}$.
Since $y\in B_{i-1,j-1,k-1} $ and
$y'\in B_{i+1,j-1,k-1}$, it follows that
$d_G(y, v)=d_G(y',v)=j-1$
and
$d_G(y, w)=d_G(y',w)=k-1$.
From Theorem \ref{th-EM},
$xy$ is not monitored by $v,w$.
Since
$y\in B_{i-1,j-1,k-1}$ and
$y''\in B_{i-1,j',k'}$,
it follows that
$d_G(y, u)=d_G(y'',u)=i-1$,
and hence $xy$ is not monitored by $u$, and so
$xy$ is not monitored by $u,v,w$,
a contradiction.
\end{proof}

From Claim \ref{claim5},
$(5)$ holds.

\begin{claim}\label{claim6}
For any $i,j,k\in \{0,1,2,\ldots,diam(G)\}$ and any $x \in  B_{i,j,k} $,  there is no $4$-star $K_{1,4}$ with edge set
$E(K_{1,4})= \{yx, z_1x, z_2x,z_3x\}$
such that $y\in B_{i-1,j-1,k-1}$,
\begin{align*}
z_1 \in &
B_{i-1,j-1,k+1}
\cup B_{i-1,j,k+1}
\cup B_{i-1,j+1,k-1}\cup B_{i-1,j+1,k}
\cup B_{i-1,j+1,k+1}, \\
z_2 \in&
B_{i-1,j-1,k+1}
\cup B_{i,j-1,k+1}
\cup B_{i+1,j-1,k-1}
\cup B_{i+1,j-1,k}
\cup B_{i+1,j-1,k+1}, \\
z_3  \in&
B_{i-1,j+1,k-1}
\cup B_{i,j+1,k-1}
\cup B_{i+1,j-1,k-1}
\cup B_{i+1,j,k-1}
\cup B_{i+1,j+1,k-1},
\end{align*}
\end{claim}
\begin{proof}
Assume, to the contrary,
that
$x \in B_{i ,j ,k } $ has four neighbors $y,z_1,z_2,z_3$ satisfying the conditions of this claim. Then $d_G(y,u)=d_G(z_1,u)=i-1$. From Theorem \ref{th-EM},
$xy$ can not be monitored by $u$.
Similarly, since
$d_G(y,v)=d_G(z_2,v)=j-1$, it follows from Theorem \ref{th-EM} that
$xy$ can not be monitored by $v$. Similarly, since $d_G(y,w)=d_G(z_3,w)=k-1$,
it follows that $xy$ can not be monitored by $w$, a contradiction.
\end{proof}

From Claim \ref{claim6},
$(6)$ holds.

\begin{claim}\label{claim7}
There is no $P_4^{+}$ with vertex set
$\{z_1, z_2, x, y, z_3\}$
and edge set $\{z_1x, z_2x, xy, yz_3\}$ such that $x\in B_{i,j,k}$, $y\in B_{i-1,j+1,k-1}$, and
\begin{align*}
z_1\in&
B_{i-1,j-1,k-1}
\cup B_{i-1,j-1,k}
\cup B_{i-1,j-1,k+1}
\cup B_{i-1,j,k+1}
\cup B_{i-1,j+1,k-1}
\cup B_{i-1,j+1,k}
\cup B_{i-1,j+1,k+1},\\
z_2 \in&
B_{i-2,j,k-2}
\cup B_{i-2,j,k-1}
\cup B_{i-2,j,k}
\cup B_{i-1,j,k-2}
\cup B_{i-1,j,k}
\cup B_{i-1,j,k-1}
\cup B_{i,j,k-2}
\cup B_{i,j,k}
\cup B_{i,j,k-1}\\
z_3\in&
B_{i-1,j-1,k-1}
\cup B_{i-1,j+1,k-1}
\cup B_{i,j-1,k-1}
\cup B_{i,j,k-1}
\cup B_{i,j+1,k-1}
\cup B_{i+1,j-1,k-1}
\cup B_{i+1,j,k-1}\\
&\cup B_{i+1,j+1,k-1}.
\end{align*}
\end{claim}
\begin{proof}
Assume, to the contrary, that there is $P_4^{+}$ satisfying the conditions of this claim.
Since $d_G(y,u)=d_G(z_1,u)=i-1$, $d_G(x,v)=d_G(z_2,v)=j$ and $d_G(y,w)=d_G(z_3,w)=k-1$, it follows from Theorem \ref{th-EM} that
$xy$ can not be monitored by $u,v,w$, respectively,
a contradiction.
\end{proof}

From Claim \ref{claim7},
$(7.1)$ holds. Similarly, we can prove that $(7.2)$ holds.

\begin{claim}\label{claim8}
There is no $3$-star $K_{1,3}$ with vertex set
$\{z_1, z_2, x, y\}$
and edge set $\{xy, xz_1, xz_2\}$
such that $x=B_{i,j,k}$,
$y=B_{i,j-1,k-1}$, and
\begin{align*}
z_1\in&
B_{i-1,j-1,k}
\cup B_{i-1,j-1,k+1}
\cup B_{i, j-1, k}
\cup B_{i,j-1,k+1}
\cup B_{i+1,j-1,k}
\cup B_{i+1,j-1,k}
\cup B_{i+1,j-1,k+1},\\
z_2\in&
B_{i-1,j-1,k-1}
\cup B_{i-1,j,k-1}
\cup B_{i-1,j+1,k-1}
\cup B_{i,j,k-1}
\cup B_{i,j+1,k-1}
\cup B_{i+1,j,k-1}
\cup B_{i+1,j+1,k-1}.
\end{align*}
\end{claim}
\begin{proof}
Assume, to the contrary, that there is a $K_{1,3}$ such that
$V(K_{1,3})=\{z_1, z_2, x, y\}$
and $E(K_{1,3})=\{xy, xz_1, xz_2\}$. Since
$d_G(x,u)=d_G(y,u)=i$, $d_G(y,v)=d_G(z_1,v)=j-1$ and $d_G(y,w)=d_G(z_2,w)=k-1$, it follows from Theorem \ref{th-EM} that $xy$ is not monitored by $u,v,w$,
a contradiction.
\end{proof}

From Claim \ref{claim8},
$(8)$ holds.

Conversely, we assume that there exists three vertices $u,v,w$ in $G_b$ such that all of the conditions $(1)$-$(8)$ holds in $G_b$.
It suffices to prove that $\{u,v,w\}$ is a distance-edge-monitoring set in $G_b$, and hence $dem(G) = 3$. Let $xy$ be any
edge of $G$ with $x\in B_{i, j, k} $.
Since $(1)$ holds, it follows that
$y\notin B_{i, j, k} $. Then we have the following cases:

\begin{case}\label{Case:1}
$y \in B_{i, j-1, k} $ or
$y \in B_{i-1, j, k} $ or
$y \in B_{i, j, k-1} $ or
$y \in B_{i, j+1, k} $ or
$y \in B_{i, j, k+1} $ or
$y \in B_{i+1, j, k} $.
\end{case}

For $x\in B_{i, j, k} $ and $y \in B_{i, j-1, k}$, we assume that $xy$ can not be monitored by $\{u,v,w\}$.
Then there is a path $P_j$ of length $j$ from $x$ to $v$ such that $xy\notin E(P_j)$.
Let $z_2$ be the neighbor of $x$ in $P_j$.
From Lemma \ref{lem-char}, we have $z_2\in B_{i', j-1, k'} $,
where $i' \in \{i-1, i, i+1\}$
and $k' \in \{k-1, k, k+1\}$, which contradicts to the condition $(3.1)$.
Suppose that $x\in B_{i, j, k} $ and $y \in B_{i, j-1, k} $. Then $xy$ can be monitored by $\{u,v,w\}$. Similarly, the edges $xy$ can be also monitored by $\{u,v,w\}$, where $y \in B_{i-1, j, k} $ or
$y \in B_{i, j-1, k} $ or
$y \in B_{i, j, k-1} $ or
$y \in B_{i, j+1, k+1} $ or
$y \in B_{i+1, j, k+1} $ or
$y \in B_{i+1, j+1, k} $.

\begin{case}\label{Case:2}
$y \in  B_{i-1, j-1, k-1} $ or $y \in B_{i+1, j+1, k+1} $.
\end{case}

For $x\in B_{i, j, k} $ and $y \in  B_{i-1, j-1, k-1}$, we assume that $xy$ is not monitored by $\{u,v,w\}$. Then there exists a path $P_i,P_j,P_k$ of length $i,j,k$ from $x$ to $u,v,w$ such that $xy\notin E(P_i)$ and $xy\notin E(P_j)$ and $xy\notin E(P_k)$, respectively.
Let $z_1,z_2,z_3$ be the neighbors of $x$ in $P_i,P_j,P_k$, respectively. Then $z_1 \in B_{i-1, j', k'} $ and $z_2 \in B_{i'', j-1, k''} $ and $z_3 \in B_{i''', j''', k-1} $.
where $i'',i'''\in \{i-1,i,i+1\}$,
$j',j'''\in \{j-1,j,j+1\}$ and
$k',k'' \in \{k-1, k, k+1\}$.
Since $x\in B_{i, j, k} $ and $y \in  B_{i-1, j-1, k-1}$, it follows from the conditions $(2)$ and $(3.2)$ that
for any $z_i\in N(x)$ ($1 \leq i \leq 3$), it follows that
$z_1 \notin \{
B_{i-1, j-1, k-1} ,
B_{i-1, j-1, k} ,
B_{i-1, j, k-1} ,
B_{i-1, j, k}
\}$,
$z_2 \notin \{$
$B_{i-1, j-1, k-1} $,
$B_{i-1, j-1, k} $,
$B_{i, j-1, k-1} $,
$B_{i, j-1, k} $
$\}$,
$z_3 \notin \{$
$B_{i-1, j-1, k-1} $,
$B_{i-1, j, k-1} $,
$B_{i, j-1, k-1} $,
$B_{i, j, k-1} $
$\}$,
and hence there is a $4$-star with edge set $\{yx, z_1x, z_2x,z_3x\}$
such that $y\in B_{i-1,j-1,k-1}$,
\begin{align*}
z_1 \in &
B_{i-1,j-1,k+1}
\cup B_{i-1,j,k+1}
\cup B_{i-1,j+1,k-1}\cup B_{i-1,j+1,k}
\cup B_{i-1,j+1,k+1}, \\
z_2 \in&
B_{i-1,j-1,k+1}
\cup B_{i,j-1,k+1}
\cup B_{i+1,j-1,k-1}
\cup B_{i+1,j-1,k}
\cup B_{i+1,j-1,k+1}, \\
z_3  \in&
B_{i-1,j+1,k-1}
\cup B_{i,j+1,k-1}
\cup B_{i+1,j-1,k-1}
\cup B_{i+1,j,k-1}
\cup B_{i+1,j+1,k-1},
\end{align*}
which contradicting to the condition $(6)$.

So, $xy$ can be monitored by $\{u,v,w\}$. Similarly, the edges $xy$ can be also monitored by $\{u,v,w\}$,
where $y \in B_{i+1, j+1, k+1} $.

\begin{case}\label{Case:3}
$y \in B_{i-1,j+1,k-1} $ or
$y \in B_{i+1,j-1,k-1} $ or
$y \in B_{i-1,j-1,k+1} $ or
$y \in B_{i+1,j-1,k+1} $ or
$y \in B_{i-1,j+1,k+1} $ or
$y \in B_{i+1,j+1,k-1} $.
\end{case}

For $x\in B_{i, j, k} $ and $y \in B_{i-1,j+1,k-1} $, we assume that $xy$ is not monitored by $\{u,v,w\}$.
Then there exists a path $P_i$ of length $i$ from $x$ to $u$ such that $xy\notin E(P_i)$, and
there exist two paths $P_{j+1},P_k$ of length $j+1,k$ from $y$ to $v,w$ such that $xy\notin E(P_{j+1})\cup E(P_{k})$, respectively. Let $z_1,z_3$ be the neighbors of $x$ on the $P_i,P_k$, respectively.
In addition,
let $z_2$ be the neighbors of $y$ on the $P_{j+1}$.

Thus, there is a $5$-vertex graph $P_4^+$
with
$z_1 \in  B_{i-1,a, b} $,
$z_2\in B_{a',j, b'} $,
$x \in  B_{i,j,k} $,
$y \in  B_{i-1,j+1,k+1} $,
$z_3\in B_{a'',b'', k-1} $,
$a, b''   \in \{j-1,j, j + 1\}$,
$b \in \{k-1,k, k + 1\}$,
$a' \in \{i-2,i-1, i\}$,
$b' \in \{k-2,k-1, k\}$,
$a'' \in \{i-1,i, i+1\}$.

Since $y \in B_{i-1,j+1,k-1}$, it follows from the condition $(2)$ and $(3.3)$ that for any  $z_i\in N(x)$ ($1 \leq i \leq 3$),
we have
$z_1 \notin \{
B_{i-1, j, k-1} ,
B_{i-1, j, k}
\}$,
$z_2 \notin \{
B_{i-1, j, k-1}
\}$,
$z_3 \notin \{
B_{i-1, j, k-1} ,
B_{i, j, k-1}
\}$.
Furthermore, we have
\begin{align*}
z_1\in&
B_{i-1,j-1,k-1}
\cup B_{i-1,j-1,k}
\cup B_{i-1,j-1,k+1}
\cup B_{i-1,j,k+1}
\cup B_{i-1,j+1,k-1}
\cup B_{i-1,j+1,k}
\cup B_{i-1,j+1,k+1},\\
z_2 \in&
B_{i-2,j,k-2}
\cup B_{i-2,j,k-1}
\cup B_{i-2,j,k}
\cup B_{i-1,j,k-2}
\cup B_{i-1,j,k}
\cup B_{i,j,k-2}
\cup B_{i,j,k-1}
\cup B_{i,j,k}.\\
z_3\in&
B_{i-1,j-1,k-1}
\cup B_{i-1,j+1,k-1}
\cup B_{i,j-1,k-1}
\cup B_{i,j+1,k-1}
\cup B_{i+1,j-1,k-1}
\cup B_{i+1,j,k-1}\\
&\cup B_{i+1,j+1,k-1}.
\end{align*}
which contradicts to the condition $(7.1)$.

\begin{case}\label{Case:4}
$y \in B_{i,j-1,k-1} $ or
$y \in B_{i-1, j-1, k} $ or
$y \in B_{i-1, j, k-1} $ or
$y \in B_{i, j+1, k+1} $ or
$y \in B_{i+1, j+1, k} $ or
$y \in B_{i+1, j, k+1} $.
\end{case}

For $x\in B_{i, j, k} $ and $y \in B_{i-1,j+1,k-1} $, we assume that $xy$ is not monitored by $\{u,v,w\}$.
Then there is a path of length $j$
from $x$ to $v$, say $P_{j}$,
such that $xy\notin E(P_{j})$.
Similarly,
there is a path of length $k$, say $P_{k}$,
from $y$ to $w$ such that
$xy\notin E(P_{k})$.
Let $z_2,z_3$ be the neighbors of $x$ on the $P_{j},P_k$, respectively.
Then there is a $3$-star $K_{1,3}$ with edge set $\{xy,xz_2,xz_3\}$ such that $x \in  B_{i,j,k}$,
$y \in  B_{i,j-1,k-1}$,
$z_2 \in  B_{a,j-1, c}$,
$z_3\in B_{a', b',k-1}$, where
$a, a' \in \{i-1,i, i+ 1\}$,
$c \in \{k-1,k, k + 1\}$,
$b' \in \{j-1, j, j+1\}$.
If $y \in B_{i,j-1,k-1}$, then
for any $z_i\in N(x)$
($2 \leq i \leq 3$), it follows from
the conditions $(2)$ and $(3.4)$,
that
\begin{align*}
z_2  \notin &
B_{i-1, j-1, k-1}\cup  B_{i, j-1, k-1} \cup  B_{i, j-1, k}\cup  B_{i+1, j-1, k-1},\\
 z_3  \notin &
B_{i-1, j-1, k-1}\cup  B_{i, j-1, k-1} \cup  B_{i, j, k-1}\cup  B_{i+1, j-1, k-1},
\end{align*}
and hence
\begin{align*}
x \in &  B_{i,j,k}, y \in  B_{i,j-1,k-1}, \\
z_2 \in &
B_{i-1,j-1,k}
\cup  B_{i-1,j-1,k+1}
\cup  B_{i,j-1,k+1}
\cup  B_{i+1,j-1,k}
\cup  B_{i+1,j-1,k+1},\\
z_3 \in &
B_{i-1,j,k-1}
\cup B_{i-1,j+1,k-1}
\cup  B_{i,j+1,k-1}
\cup  B_{i+1,j,k-1}
\cup  B_{i+1,j+1,k-1},
\end{align*}
which contradicts to the condition $(8)$.

Similarly,  if $x\in B_{i, j, k} $ and
$y \in B_{i,j-1,k-1} $,
then $xy$ can be monitored by $\{u,v,w\}$.
By the same method, we can prove that the edges $xy$ can be also monitored by $\{u,v,w\}$, where
$y \in B_{i-1, j-1, k} $ or
$y \in B_{i-1, j, k-1} $ or
$y \in B_{i, j+1, k+1} $ or
$y \in B_{i+1, j+1, k} $ or
$y \in B_{i+1, j, k+1} $.

\begin{case}\label{Case:5}
$y \in B_{i,j-1,k+1} $ or
$y \in B_{i-1, j, k+1} $ or
$y \in B_{i, j+1, k-1} $ or
$y \in B_{i+1, j-1, k} $ or
$y \in B_{i+1, j, k-1} $ or
$y \in B_{i-1, j+1, k} $.
\end{case}

For $x\in B_{i, j, k}$
and $y \in B_{i,j-1,k+1}$,
we assume that $xy$ is not monitored by $\{u,v,w\}$.
Then
there is a path of length $k+1$ from $y$ to $w$, say $P_{k+1}$, such that
$xy\notin E(P_{k+1})$, and
there is a path of length $j$ from $x$ to $v$, say $P_{j}$,
such that
$xy\notin E(P_{j})$.
Let $z_3,z_2$ be the neighbors of $y,x$ on the $P_{k+1},P_{j}$, respectively.
Thus, there is a $4$-path $P_4$ with
$V(P_4)=\{z_2, x, y, z_3\}$
and
$E(P_4)=\{z_2x, xy, yz_3\}$ such that $x \in  B_{i,j,k}$,
$y \in  B_{i,j-1,k+1}$,
$z_2\in B_{a,j-1, b} $, and
$z_3\in B_{a',b', k} $, where
$a, a'\in \{i-1,i, i+1\}$,
$b \in \{k-1,k, k+1\}$,
$b' \in \{j-2,j-1, j\}$.

From the conditions $(2)$
and $(3.5)$, we have
$y \in B_{i,j-1,k+1}$, and for any $z_i\in N(x)$
($2 \leq i \leq 3$), we have
$z_3 \notin \{
B_{i, j-1, k} ,
\}$,
$z_2 \notin \{
B_{i, j-1, k}
\}$, and hence $x=B_{i,j,k} $,
$y=B_{i,j-1,k+1}$,
\begin{align*}
z_2 \in &
B_{i-1,j-1,k-1}
\cup B_{i-1,j-1,k}
\cup B_{i-1,j-1,k+1}
\cup B_{i,j-1,k-1}
\cup B_{i,j-1,k+1}
\cup B_{i+1,j-1,k-1}
\cup B_{i+1,j-1,k}\\
& \cup B_{i+1,j-1,k+1},
\\
z_3 \in &
B_{i-1,j-2,k}
\cup B_{i-1,j-1,k}
\cup B_{i-1,j,k}
\cup B_{i,j-2,k}
\cup B_{i,j,k}
\cup B_{i+1,j-2,k}\\
& \cup B_{i+1,j-1,k}
\cup B_{i+1,j,k}.
\end{align*}
which contradicts the condition $(4.3)$.
If $x\in B_{i, j, k} $ and
$y \in B_{i,j-1,k+1} $,
then $xy$ can be monitored by $\{u,v,w\}$.
Similarly, the edges $xy$ can be also monitored by $\{u,v,w\}$, where
$y \in B_{i-1, j, k+1} $ or
$y \in B_{i, j+1, k-1} $ or
$y \in B_{i+1, j-1, k} $ or
$y \in B_{i+1, j, k-1} $ or
$y \in B_{i-1, j+1, k} $.

If $x\in B_{i,j,k}$,
from it follows Lemma \ref{lem-char},
that
$y \in T$,
where
$$
T=\left\{ B_{i',j',k'}\,|\,i' \in \{i-1,i,i+1\},
j' \in \{j-1,j,j+1\},
k' \in \{k-1,k,k+1\}\right\}.
$$

From the above cases, the vertex set $B_{i,j,k}(u,v,w)$ has the arbitrariness.
Then the $xy$ in $E(G_b)$
can be monitored by $\{u,v,w\}$, and hence
$\{u,v,w\}$ is a distance-edge-monitoring set in $G_b$, and so $dem(G) = 3$.
\end{proof}

\section{Conclusion}

In this paper, we have continued the study of {\it distance-edge-monitoring sets}, a new graph parameter recently introduced by Foucaud {\it et al.}~\cite{FKKMR21}, which is useful in the area of network monitoring. In particular, we have given upper and lower bounds on the parameters $P(M,e)$,
$EM(x)$, $dem(G)$, respectively, and extremal graphs attaining the bounds were characterized.
We also characterized the graphs with $dem(G)=3$.

For future work, it would be interesting to study distance-edge monitoring sets in further standard graph classes, including pyramids, Sierpi\'nki-type graphs, circulant graphs, graph products, or line graphs. In addition, characterizing the graphs with $dem(G)=n-2$ would be of interest, as well as clarifying further the relation of the parameter $dem(G)$
to other standard graph parameters, such as arboricity, vertex cover number and feedback edge set number.

\end{document}